\documentclass[reqno]{amsart}
\usepackage{srcltx}
\usepackage{eurosym}
\usepackage{mathtools}
\usepackage{amsmath}
\usepackage{amsfonts}
\usepackage{amssymb}
\usepackage{amsthm}
\usepackage{graphicx}
\usepackage{mathrsfs}
\usepackage{xcolor}
\usepackage{exscale}
\usepackage{latexsym}
\usepackage[colorlinks,plainpages=true,pdfpagelabels,hypertexnames=true,colorlinks=true,pdfstartview=FitV,linkcolor=blue,citecolor=red,urlcolor=black]{hyperref}
\PassOptionsToPackage{unicode}{hyperref}
\PassOptionsToPackage{naturalnames}{hyperref}
\usepackage{enumerate}
\usepackage[shortlabels]{enumitem}
\usepackage{bookmark}
\usepackage{wasysym}
\usepackage{esint}
\usepackage[ddmmyyyy]{datetime}
\usepackage[margin=3cm]{geometry}
\makeatletter
\g@addto@macro\normalsize{%
  \setlength\abovedisplayskip{2pt}
  \setlength\belowdisplayskip{3pt}
  \setlength\abovedisplayshortskip{4pt}
  \setlength\belowdisplayshortskip{4pt}
}
\numberwithin{equation}{section}
\everymath{\displaystyle}
\newtheorem{theorem}{Theorem}[section]
\newtheorem{lemma}[theorem]{Lemma}
\newtheorem{corollary}[theorem]{Corollary}
\newtheorem{proposition}[theorem]{Proposition}

\newtheorem{definition}[theorem]{Definition}
\newtheorem{remark}[theorem]{Remark}        
  
\numberwithin{equation}{section}
\newcommand{\pa}{\partial}

\newcommand{\pp}{\mathcal{P}}
\newcommand{\tom}{\tilde{\om}}

\makeatletter
\newcommand{\vo}{\vec{o}\@ifnextchar{^}{\,}{}}
\makeatother

\def\YYint#1#2#3{{\setbox0=\hbox{$#1{#2#3}{\iint}$}
    \vcenter{\hbox{$#2#3$}}\kern-.50\wd0}}

\def\XXint#1#2#3{{\setbox0=\hbox{$#1{#2#3}{\int}$}
    \vcenter{\hbox{$#2#3$}}\kern-.50\wd0}}

\makeatletter
\def\namedlabel#1#2{\begingroup
   \def\@currentlabel{#2}%
   \label{#1}\endgroup
}
\makeatother
\makeatletter
\newcommand{\rmh}[1]{\mathpalette{\raisem@th{#1}}}
\newcommand{\raisem@th}[3]{\hspace*{-1pt}\raisebox{#1}{$#2#3$}}
\makeatother

\newcommand{\redref}[2]{\texorpdfstring{\protect\hyperlink{#1}{\textcolor{black}{(}\textcolor{red}{#2}\textcolor{black}{)}}}{}}
\newcommand{\redlabel}[2]{\hypertarget{#1}{\textcolor{black}{(}\textcolor{red}{#2}\textcolor{black}{)}}}
\newcommand{\descitem}[2]{\item[{\bf #1}]\label{#2}}
\newcommand{\descref}[2]{\hyperref[#1]{\textnormal{\textcolor{black}{(}\textcolor{blue}{\bf #2}\textcolor{black}{)}}}}

\newcommand{\dref}[2]{\hyperref[#1]{\textcolor{black}{\bf #2}}}
\newcommand{\tf}{\tilde{f}}
\newcommand{\tg}{\tilde{g}}
\newcommand{\tF}{\tilde{F}}
\newcommand{\tL}{\tilde{L}}

\newcommand{\tde}{{\tilde{\delta}}}
\newcommand{\tal}{\tilde{\al}}
\newcommand{\mr}{{{r}_0}}

\newcommand\RR{\mathbb{R}}

\newcommand\NN{\mathbb{N}}
\newcommand{\al}{\alpha}
\newcommand{\be}{\beta}

\newcommand{\de}{\delta}
\newcommand{\ve}{\varepsilon}

\newcommand{{\ka}}{\kappa}

\newcommand{\om}{\omega}
\newcommand{\la}{\lambda}

\newcommand{\Ga}{\Gamma}

\newcommand{\Th}{\Theta}

\newcommand{\Om}{\Omega}

\newcommand{\La}{\Lambda}

\DeclareMathOperator{\dini}{Dini}
\DeclareMathOperator{\loc}{loc}
\DeclareMathOperator{\holder}{Hdr}
\DeclareMathOperator{\md}{{\bf Mod}}
\DeclareMathOperator{\one}{{\bf I}}
\DeclareMathOperator{\two}{{\bf II}}
\DeclareMathOperator{\three}{{\bf III}}
\DeclareMathOperator{\four}{{\bf IV}}

\newcommand{\iprod}[2]{\langle #1,  #2\rangle}

\newcommand{\abs}[1]{\left| #1\right|}
\newcommand{\lbr}[1][(]{\left#1}
\newcommand{\rbr}[1][)]{\right#1}
\newcommand{\txt}[1]{\qquad \text{#1} \quad}
\usepackage[onehalfspacing]{setspace}
\usepackage[titletoc,toc,page]{appendix}

\begin{document}

\title{Borderline regularity for fully nonlinear equations in Dini domains}

\author{Karthik Adimurthi}
\address{Department of Mathematical Sciences, Seoul National University, Seoul 08826, Korea.}
\email{karthikaditi@gmail.com and kadimurthi@snu.ac.kr}
\thanks{The first author was supported in part by the National Research Foundation of Korea grant NRF-2015R1A4A1041675.}

\author{Agnid Banerjee}
\address{Tata Institute of Fundamental Research - Centre for Applicable Mathematics, Bangalore, Karnataka, India 560065.}
\email{agnid@tifrbng.res.in and agnidban@gmail.com}





\begin{abstract}
In this paper, we prove borderline gradient continuity of viscosity solutions to  Fully nonlinear elliptic  equations at the boundary of a $C^{1,\dini}$-domain. These results (see Theorem \ref{main_thm}) are a sharpening of the boundary gradient estimate proved in \cite{MR2853528} following the borderline interior gradient regularity estimates established  in \cite{MR3169795}. We however mention that, differently  from the approach in \cite{MR3169795}  which is  based on $W^{1,q}$ estimates, our  proof is  slightly more geometric and is  based on compactness arguments inspired by the techniques in the fundamental works of Caffarelli as in  \cite{MR1005611,MR1038360,MR1038359}.
\end{abstract}

\maketitle

\tableofcontents
%

\section{Introduction}
The aim of this paper is to   obtain pointwise gradient continuity estimates  upto the boundary for viscosity solutions of 
\begin{equation}
\label{main}
\left\{\begin{array}{ll}
F(x,D^2 u ) = f & \text{in}\ \Om,\\
u = g & \text{on} \ \pa \Om,
\end{array}\right.
\end{equation}
under minimal  regularity  assumptions  on $f,g$ and  $\pa \Om$.

\medskip

The fundamental role of these regularity estimates in the theory of elliptic and parabolic partial differential equations is well known.  In order to put our results in the correct historical perspective, we note that in 1981,  E. Stein in his visionary work \cite{MR607898} showed the following "limiting" case of Sobolev embedding theorem.
\begin{theorem}\label{stein}
Let $L(n,1)$ denote the standard Lorentz space, then the following implication holds:
\[\nabla v \in L(n,1) \ \implies \ v\  \text{is continuous}.\]  
\end{theorem}
The Lorentz space $L(n,1)$ appearing in Theorem \ref{stein}  consists of those measurable functions $g$ satisfying the condition
\[
\int_{0}^{\infty} |\{x: g(x) > t\}|^{1/n} dt < \infty.
\]
Theorem \ref{stein} can be regarded as the limiting case of Sobolev-Morrey embedding that asserts
\[
\nabla v \in L^{n+\ve} \implies v \in C^{0, \frac{\ve}{n+\ve}}.
\]
Note that indeed $L^{n+\ve} \subset L(n, 1) \subset L^{n}$ for any $\ve>0$ with all the inclusions being strict.  Now Theorem \ref{stein} coupled with the standard Calderon-Zygmund theory  has the following interesting consequence.
\begin{theorem}\label{st}
$\Delta u \in L(n,1) \implies \nabla u$ is continuous.
\end{theorem} 
The analogue of Theorem \ref{st}  for general nonlinear and possibly degenerate elliptic and parabolic equations has become accessible not so long ago through a rather sophisticated and powerful nonlinear potential theory (see for instance \cite{MR2823872,MR2900466,MR3174278} and the references therein).  The first breakthrough in this direction came up in the work  of Kuusi and Mingione in \cite{MR3004772} where they showed that the analogue of Theorem \ref{st}  holds for operators modelled after the $p$-laplacian. Such a result  was subsequently generalized to $p$-laplacian type systems by the same authors in  \cite{MR3247381}.

Since then, there has been several generalizations of Theorem \ref{st} to operators with other kinds of nonlinearities and in context of  fully nonlinear elliptic equations, the  analogue of Theorem \ref{st} has been established by Daskalopoulos-Kuusi-Mingione in  
\cite{MR3169795}. More precisely, they showed that (see Theorem 1.1 in \cite{MR3169795})
\begin{theorem}\label{dkm}
Let $u$ be a viscosity solution to
\begin{equation}\label{fl}
F(x,D^2u)= f\   \txt{in} \Om,
\end{equation}
where $F$ is uniformly elliptic fully nonlinear operator and $f \in L(n, 1)$. Then there exists $\theta \in (0, 1)$ depending only on $n$ and the ellipticity constants of $F$  such that if $F(.)$ has $\theta$-BMO coefficients, then $Du$ is continuous in the interior of $\Om$. 
\end{theorem}

It turns out that the key to the nonlinear theory as observed in \cite{MR3004772} is to consider the following  modified  $L^{q}$ version of the classical Riesz potential:
\begin{equation}\label{rz}
{\bf \tilde  I}^{f}_{q}(x, r)= \int_{0}^{r} \left( \fint_{B_{\rho}(x)} |f(y)|^q dy \right)^{1/q} d\rho,
\end{equation}
 and then getting gradient $L^{\infty}$ as well as moduli of continuity  estimates  in terms of this modified Riesz potential, which is analogous to  the classical linear theory where similar estimates are known in terms of the truncated Riesz potential. In the context of fully nonlinear elliptic equations as in Theorem \ref{dkm} above,  the authors show that the following estimate holds
\begin{equation}\label{md}
|\nabla u(x_1) - \nabla u(x_2)| \leq c\|Du\|_{L^{\infty}(\Om)} |x_1-x_2|^{\alpha(1-\delta)} + c {\bf \tilde  I}^{f}_{q}(x, 4 |x_1-x_2|^{\delta})
\end{equation}
where $\alpha, \delta$ depends on $n$, $q$ and the ellipticity constants of $F$.

Estimate \eqref{md}  from  \cite{MR3169795} is obtained by a  delicate combination of   $W^{1,q}$  estimates for fully nonlinear equations established in \cite{MR1606359} with a certain modified  Morrey-Campanato type argument. Over here, the reader should note that the success of such a small perturbation type argument relies crucially on  intricate scaling properties of the equation  in Theorem \ref{dkm} (which for $f \in L(n,1)$ is  scaling "critical"  as the reader will observe in our work later on) and also on the fact that at small enough scales, such an equation  can be regarded as  a small perturbation of 
\[
F(D^2u)=0,
\]
for which apriori $C^{1, \alpha}$ estimates  are known. (see for instance \cite{MR1005611}).   It turns out that if $f \in L(n ,1)$, then
\begin{equation}\label{convergence}
{\bf \tilde  I}^{f}_{q}(x, r) \to 0\txt{as} r \to 0,
\end{equation}
whenever $q < n$, which combined with the estimate \eqref{md} gives that $\nabla u$ is continuous. 

\medskip

These recent results have provided us with  a natural motivation to investigate  the validity of  similar gradient  continuity estimates upto the boundary for solutions to  \eqref{fl}  in the borderline situation as in Theorem \ref{dkm} (i.e., with $f \in L(n,1)$) and with minimal regularity assumptions on the boundary and the boundary datum.   Our main result (see Theorem \ref{main_thm}) can be thought of as the boundary analogue of Theorem \ref{dkm} which was established in \cite{MR3169795}.  More precisely, in the boundary situation as in \eqref{main}, we show that if $\Om$ is $C^{1, \dini}$ and $g$ is $C^{1, \dini}$, then $Du$ is continuous upto $\pa \Om$  with a modulus of continuity similar to that  in Theorem \ref{dkm}, in particular an estimate of the form \eqref{md} holds upto the boundary. Note that standard results on gradient continuity of solutions  requires $\pa \Om \in C^{1,\dini}$ (see for instance \cite{MR818099}) and $C^{1}$ regularity     is not true in general for $C^{1}$ domains where the solution may even fail to be Lipschitz upto the boundary (see for instance \cite{MR0437947}). Therefore  in that sense, our regularity assumptions on Dirichlet boundary conditions are in a sense, optimal. 

\medskip

The reader should note that  in order to obtain an estimate similar to  \eqref{md}  in our situation from which gradient continuity follows thanks to the convergence in  \eqref{convergence}, we follow an approach which  is somewhat  different from the approach used in \cite{MR3169795} and therefore our work gives a slightly different  viewpoint in the interior case  as well.  Our method  is based on the adaptation of  compactness arguments  and  is independent of the  $W^{1,q}$ estimates which is crucially used in \cite{MR3169795} in order to establish \eqref{md}. We note that such geometric compactness arguments  have their roots in the seminal work of Caffarelli (see \cite{MR1005611}) and   is based on pointwise  affine approximation  of the solution at dyadic scales which is achieved  in our situation by suitable rescalings that are partially inspired by those used in \cite{MR3169795} and   by  appropriately comparing  our boundary value problem with a  relatively smooth Dirichlet problem.  We would like to mention as well that although our work is  inspired   by some of the earlier works mentioned above, it has nonetheless required some delicate  adaptations  in our setting which is complicated by the presence of the Dirichlet condition.  For instance,  in order to ensure that  our compactness lemma  in Section \ref{section4} can be applied, we have to additionally ensure smallness of the boundary datum at each step of iteration.  The reader can see from the analysis involved in the proof of Lemma \ref{lemma4.7}  that this requires some subtle work  in our Dirichlet situation because of additional  moduli of continuities involved unlike the interior case. We also note that  unlike what is conventionally done in the divergence form theory,   the boundary cannot be flattened in our situation to begin with    because of the lower regularity assumption on $\Om$ and the fact that our equation has  non-divergence structure.  In that sense, our techniques are also partially inspired by that in  the  recent paper  \cite{arXiv:1804.06697}, which is on boundary Schauder estimates on Carnot groups where the boundary cannot be flattened either. 
\medskip

Finally, we   describe  a related  boundary regularity   result that has been previously obtained by  Ma-Wang in their interesting paper \cite{ MR2853528}. In \cite{MR2853528}, the authors establish  the gradient continuity of solutions to \eqref{main} upto the boundary of a $C^{1,\dini}$ domain $\Om$ under the assumption that

\begin{equation}\label{cv}
\int_{0}^{r} \left( \fint_{B_{\rho}(x)} |f(y)|^n dy \right)^{1/n} d\rho =:{\bf \tilde  I}^{f}_{n}(x, r)\to 0\ \txt{as} r \to 0,
\end{equation}
 or equivalently under the assumption that the convergence in \eqref{convergence} holds for $q=n$. Now  for an arbitrary  function $f \in L(n,1)$, the convergence in \eqref{convergence} doesn't hold  when $q=n$ and hence the  result from \cite{MR2853528}  doesn't cover our regularity result. Therefore our main result  is a  true  sharpening of the result in \cite{ MR2853528}. We note that the method in \cite{ MR2853528}, which in turn is inspired by some of the the fundamental works of Wang in \cite{ MR1139064} is quite different from  ours and  makes clever use of  barriers,  Alexandroff-Bakelman-Pucci type  maximum principle  and $\dini$-continuity of the normal at the boundary  using which the authors obtain  appropriate estimates at each  iterative step (see the proof of \cite[Lemma 3.1]{ MR2853528}).  Because of  the use of Alexandroff-Bakelman-Pucci maximum principle, their  estimate  relies crucially  on the $L^{n}$ norm  of $f$ at each step and that is precisely why  their gradient continuity estimate     depend in an essential way on the   convergence of the quantity  as  in \eqref{cv}  which involves the $L^{n}$ norm of $f$ at each scale. This heuristics shows that   the approach in \cite{ MR2853528}  cannot  be modified to prove our borderline  regularity result at the boundary.

 The paper is organized as follows. In section \ref{prelim}, we introduce certain relevant   notions and gather some known results. In section \ref{section3},  we state our main result. In Section \ref{section4},  we first establish a basic compactness Lemma and then consequently establish uniform affine approximation of the solution at the boundary at dyadic scales (see Lemma \ref{lemma4.7})  and finally in Section \ref{section5}, we prove  main Theorem \ref{main_thm}.

\section{Preliminaries}
\label{prelim}
In this section, we shall collect all the preliminary material that will be used in the subsequent sections. 
\subsection{Fully Nonlinear equations}
In this subsection, let us recall some well known definitions and properties of Fully nonlinear equations. This  subsection is taken from \cite{MR1351007} (see also \cite{MR2486925}).
\begin{definition}
\label{def_unif_ell}
Let $M(n)$ be the set of all symmetric $n\times n$ matrices equipped with the order $M \leq N$ iff $M-N$ is positive semi-definite. Any function $F: \RR^n \times M(n)  \longmapsto \RR$ is said to be uniformly elliptic if there exists constants $0< \La_0 < \La_1 < \infty$ such that for almost every $x \in \RR^n$, the following holds:
\begin{equation*}
\pp^-(M-N) \leq F(x,N) - F(x,N) \leq \pp^+(M-N),
\end{equation*}
where $\pp^-$ and $\pp^+$ are the standard Pucci's extremal operators defined as
\begin{equation*}
\pp^-(M):= \La_0 \sum_{\mu_j>0} \mu_j + \La_1 \sum_{\mu_j<0} \mu_j, \qquad  \pp^+(M):= \La_1 \sum_{\mu_j>0} \mu_j + \La_0 \sum_{\mu_j<0} \mu_j,
\end{equation*}
where $\{\mu_j\}_{j=1}^{n}$ are the eigenvalues of $M$.
\end{definition}

\begin{definition}
\label{def_visc}
Let $F(x,M)$ be continuous in $M$ and measurable in $x$ and we assume $f \in L^q_{\loc}$ for some $q > \frac{n}{2}$. A continuous function $u$ is an $W^{2,q}$-\emph{viscosity subsolution}(\emph{supersolution}) of \eqref{main} if for all $\phi \in W^{2,q}(B_r(x_0))$ with $B_r(x_0) \subset \Om$ and any $\ve >0$ satisfying the bound 
\begin{gather*}
\phantom{(}F(x,D^2\phi(x)) \leq f(x) - \ve\phantom{)}  \qquad \text{almost everywhere in } \ B_r(x_0) \\
(F(x,D^2\phi(x)) \geq f(x) + \ve) \qquad \text{almost everywhere in } \ B_r(x_0),
\end{gather*}
implies $u-\phi$ cannot attain a local maximum (minimum) at $x_0$. The function $u$ is called  $W^{2,q}$-\emph{viscosity solution} if $u$ is both a subsolution and supersolution. 
\end{definition}

\subsection{Modulus of continuity}
In this subsection, we shall recall some of the properties of modulus of continuity functions. 
\begin{definition}\label{def_modulus}
A function $\Psi(t)$ for $0 \leq t \leq S_0$ is said to be a modulus of continuity if the following properties are satisfied:
\begin{itemize}
\item $\Psi(t) \rightarrow 0$ as $t \searrow 0$.
\item $\Psi(t)$ is positive and increasing as a function of $t$.
\item $\Psi(t)$ is sub-additive, i.e., $\Psi(t_1 + t_2) \leq \Psi(t_1) + \Psi(t_2).$
\item $\Psi(t)$ is continuous.
\end{itemize}
\end{definition}

We now define the notion of $\dini$-continuity:
\begin{definition}
\label{def_dini}
 Let  $f : \RR^n \rightarrow \RR$ be a function and define the following modulus of continuity:
 \begin{equation*}
 \om_f(t):= \sup_{|x-y| \leq t} |f(x) - f(y)|.
 \end{equation*}
 We then say $f$ is  $\dini$-continuous if 
 \begin{equation}\label{dini_form}
 \int_0^1 \frac{\om_f(t)}{t} \ dt < \infty.
 \end{equation}
\end{definition}

From \cite[Page 44]{MR0213785}, we see that any continuous, increasing function $\Psi(t)$ on the interval $[0,S_0]$ which satisfies $\Psi(0)=0$ is a modulus of continuity if it is concave. From this, we have the following important result proved in \cite[Theorem 8]{MR0213785}:
\begin{theorem}
\label{mod_concave}
For each modulus of continuity $\Psi(t)$ on  $[0,S_0]$, there is a concave modulus of continuity  $\tilde{\Psi}(t)$ with the property
\[
\Psi(t) \leq \tilde{\Psi}(t) \leq 2 \Psi(t) \txt{for all} t \in [0,S_0].
\]
\end{theorem}
We will also need the following definition which captures a certain monotonicity property of the modulus of continuity.
\begin{definition}
\label{eta_decreasing}
Given $\eta \in (0,1]$, we say that a modulus $\Psi$ is $\eta$-decreasing if the following holds:
\begin{equation*}
\frac{\Psi(t_1)}{t_1^{\eta}} \geq \frac{\Psi(t_2)}{t_2^{\eta}} \txt{for all} t_1 \leq t_2.
\end{equation*}

\end{definition}

\begin{remark}
\label{remark2.7}
From \cite[Page 44]{MR0213785}, we see that any continuous, increasing function $\Psi$ on an interval $[0,S_0]$ with $\Psi(0)=0$ is a modulus of continuity if it is concave. More generally, it suffices to assume  that $\frac{\Psi(x)}{x}$ is decreasing instead of concavity for $\Psi$. 

\end{remark}

\subsection{Geometric structure}
Let us now make clear the geometric assumptions imposed on the boundary of the domain $\Om$  and on the nonlinearity.

\begin{definition}
\label{dini_domain}
We say  $\Om$ is  $C^{1, \dini}$ domain if  after translation,  rotation  and scaling, we may assume that $0 \in \pa \Om$ and  $\Om \cap B_r$ for any $r \in (0,R_0]$ is given by
\begin{equation*}
\Om \cap B_r := \left\{ (x',x_n): x_n > \Ga(x') \right\},
\end{equation*}
where $\Ga \in C^{1,\dini}$ function and $\Ga(0)=0, \nabla_{x'} \Ga(0)=0$.  In particular, $\nabla \Ga$ has $\dini$ modulus of continuity in the sense of Definition \ref{def_dini}.

\end{definition}

\begin{definition}
\label{BMO-def}
Let $F: \RR^n \times M(n)  \mapsto \RR^n$ be continuous in $x$ and define
\begin{equation*}
\Theta_F(x,y) = \Theta(x,y):= \sup_{M \in S(n) \setminus \{0\}} \frac{|F(x,M) - F(y,M)|}{\|M\|}.
\end{equation*}

We say $F$ is $\Theta_0$-BMO in $\Om$ for some $\Theta_0>0$,  if the following holds: 
\begin{equation*}
\lbr \fint_{B_r \cap \Om} \Theta(x_0,x)^n \ dx \rbr^{\frac1n} \leq \Theta_0 \qquad \text{for all } \ x_0 \in \Om \ \text{and} \ r>0.
\end{equation*}

\end{definition}

\subsection{Extension Lemma}
In this subsection, let us recall a standard extension Lemma proved in \cite[Theorem 2.2]{MR0350177} that will be used throughout the paper. For the sake of completeness, we include its proof.

\begin{lemma}\label{ext_lemma}
Let $k_0 \in \NN$ be a fixed integer and let $\Om$ be a $C^{k_0,\alpha}$ domain for some $\alpha>0$,  $f \in C^{k_0,\alpha}(\Om \cap B_1(x_0))$ be a function for some fixed $x_0 \in \pa \Om$. There exists a $C^{k_0,\alpha}$ function $\tf$ defined on $B_1(x_0)$ such that $\tf(x)=f(x)$ whenever $x \in \Om \cap B_1(x_0)$ and 
\[
\| \tf\|_{C^{k_0, \alpha}(B_1(x_0))} \leq C \|f\|_{C^{k_0, \alpha}(\Om \cap B_1(x_0))}.
\]

\end{lemma}
\begin{proof}
Let $\pa \Om \cap B_1(x_0)$ be parametrized by 
\begin{itemize}
\item $\phi(x_0)= 0 $, 
\item $\Om \cap B_1(x_0) = \left\{(x',x_n) \in B_1(x_0): \phi(x') > x_n\right\}$.
\end{itemize}
After translating and  flattening the boundary, there exists a  $C^{k_0, \alpha}$ diffeomorphism  such that
\begin{equation*}\label{diff}
\Phi: \Om \cap B_1(x_0) \mapsto B_1^+(0) \qquad \text{and} \qquad \Phi : \Om^c \cap B_1(x_0) \mapsto B_1^-(0).
\end{equation*}
Let us define the following extension function:
\[
v(x',x_n) = f \circ \Phi^{-1}(x',x_n) \txt{for all} (x',x_n) \in B_1^+(0).
\]
Since $f$ and $\Phi$ are $C^{k_0, \alpha}$ function, we must have $v \in C^{k_0, \alpha}(B_1^+(0))$.  We shall define the extension function for  $(x',x_n) \in B_1(0)$ by
\begin{equation}\label{ext}
V(x',x_n) = \left\{ \begin{array}{ll}
                      v(x',x_n) & \qquad \text{if} \ x_n \geq 0,\\
                      \sum_{i=1}^{k_0+1} c_i v\lbr x',- \frac{x_n}{i}\rbr & \qquad \text{if} \ x_n < 0,
                      \end{array}\right.
\end{equation}
where the constants $c_i$ are obtained by solving the linear system $\left\{\sum_{k=1}^{k_0+1} (-1)^j k^{-j} c_k =1\right\}_{0\leq j \leq k_0}$. From \cite[Theorem 2.2]{MR0350177}, we see that  $V \in C^{k_0, \alpha}(B_1(0))$. We now define the extension function $\tf$ by 
\[
\tf(x',x_n) = V \circ \Phi (x',x_n) \txt{for all} (x',x_n) \in B_1(x_0).
\]
It is easy to see that the extension function $\tf \in C^{k_0, \alpha}(\Om \cap B_1(x_0))$ and the following bound holds:
\[
\| \tf\|_{C^{k_0, \alpha}(B_1(x_0))} \leq C \|f\|_{C^{k_0}(\Om \cap B_1(x_0))}.
\]
This completes the proof of the Lemma.
\end{proof}

\section{Main Theorem}
\label{section3}

Let us now state the main theorem that we will prove:
\begin{theorem}
\label{main_thm}
Let $u$ be a $W^{1,q}$ viscosity solution for some $q > n-n_0$ to \eqref{main} in $\Om \cap B_1$. There exists an $\Theta \in (0,1)$ depending only on $\La_0,\La_1,n$ such that if $F$ has $\Theta$-BMO coefficients,   $f \in L(n,1)(\overline{\Om \cap B_1})$, $g$ is $C^{1,\dini}(\pa \Om)$ on a domain $\Om$ with $C^{1,\dini}$ boundary, then $\nabla u$ is continuous upto the boundary. 

In particular, for any two points $y,z \in \Om \cap B_{\frac14}$, there exists two universal constants $C_0$ and $C_1$ such that the following estimate holds: 
\begin{equation*}
|\nabla u(y) - \nabla u(z)| \leq C_0 K (C_1|y-z|).
\end{equation*}
where $K(\cdot)$ is as defined in \eqref{def_K_combined} and depends on $\dini$-modulus of $\pa \Om$,  the $\dini$-modulus of $g$ and the $L(n,1)$ character of $f$. 
\end{theorem}

\begin{remark}
Note that by a standard covering argument, we conclude that $\nabla u$ is continuous in $\overline{\Om \cap B_r}$  for any $r<1$. 
\end{remark}


\section{Some useful Lemmas}
\label{section4}

Before we begin this section, let us fix an exponent $q \in (n-n_0,n)$ where $n_0$ (denoted by $\ve$ in \cite{MR1237053}) is a small universal constant as obtained in \cite{MR1237053} such that the Krylov-Safanov type H\"older estimate  holds  for $W^{2,q}$ viscosity solutions to 
\[
F(D^2u,x)=f\txt{with} f \in L^{q}.
\]
See also \cite{MR2486925} for the analogous estimate upto the boundary. 
 \begin{definition}\label{class_bnd}
 Let $(\La_0,\La_1)$ be two fixed constants. Let $\mathfrak{F}$ denote the set of all uniformly elliptic functions $\tF$ with elliptic constants $(\La_0,\La_1)$. Furthermore, denote $\mathfrak{U}$ to be the class of all viscosity solutions $v \in C^0(B_1^+)$ solving 
 \[
 \left\{\begin{array}{ll}
\tilde{F}(x,D^2 v ) = 0 & \text{in}\ B_1^+,\\
v = 0 & \text{on} \ B_1 \cap \{x_n=0\},\\
\|v\|_{L^{\infty}(B_1^+)} \leq 1,
\end{array}\right.
 \]
 in the sense of Definition \ref{def_visc}.
 \end{definition}
The following boundary regularity was proved in \cite[Theorem 1.1]{MR3246039}:
 \begin{proposition}
 \label{bnd_reg}
 There exists an $\be = \be(n,\La_0,\La_1)\in (0,1)$ such that any solution $v \in \mathfrak{U}$ has the improved regularity $v \in C^{1,\be}(\overline{B_{\frac12}^+})$.
 \end{proposition}

 Let us now fix a constant $\al$ (see proof of Proposition \ref{prop5.1}) such that
 \begin{equation}
 \label{al_alt}
0< \al < \be. 
 \end{equation}

\subsection{Compactness Lemma}
We now state our first relevant compactness lemma at the boundary. 
\begin{lemma}
\label{lemma4.3}
Let $u$ be the viscosity solution of \eqref{main} $\Omega \cap B_1$  with $\| u\|_{L^{\infty}(\Om \cap B_1)} \leq 1$. Furthermore, suppose that $\Om$  is $C^{1,\dini}$  and furthermore assume that $\Om$  can be parametrized  as in the set  up of Definition \ref{dini_domain} with $R_0=1$  and  with $f \in L^q(\Om)$ and $g \in C^{1,\dini}(\pa \Om)$. Consider now the local problem
\begin{equation*}
\left\{
\begin{array}{ll}
F(x,D^2 u) = f & \text{in} \ \Om \cap B_1, \\
u = g & \text{on} \ \pa \Om \cap B_1,
\end{array}\right.
\end{equation*}
then given  any $\ve > 0$, there exists an $\de = \de(\La_0,\La_1,n,\ve,\dini)>0$ such that if 
\begin{gather}
\|f\|_{L^q(B_1\cap \Om)} \leq \de, \qquad \|\Ga\|_{C^{1,\dini}( B_1^{'})} \leq \de, \qquad \|g\|_{C^{0,1}(\pa \Om \cap B_1)} \leq \de, \qquad \|\Theta_F\|_{BMO} \leq \de,\label{eq4.3}
\end{gather}
then there exists a function $h \in C^{1, \be}(\overline{B_{1/2}})$ with $\be$ as obtained in Proposition \ref{bnd_reg} such that
\begin{equation*}
\label{4.4}
|u(x) - h(x)| \leq \ve \qquad \text{in} \ \Om \cap B_{\frac12}.
\end{equation*}
\end{lemma}
\begin{proof}
The proof is by contradiction and follows the strategy from \cite[Proposition 3.2]{MR2486925} (see also \cite[Lemma 2.3]{MR1606359}). Suppose the Lemma is false, then there exists an $\ve_0$ such that for any $k \in \NN$, there exists a function $f_k$, an operator $F_k$ with ellipticity constants $(\La_0,\La_1)$, boundary data $g_k$ and domains $\Om_k$ parametrized by $\Ga_k$ satisfying
\begin{equation}
\label{5.4}
 \|f_k\|_{L^q(\Om_k \cap B_1)} \leq \frac1k, \qquad \|\Ga_k\|_{C^{1,\dini}( B_1^{'})} \leq \frac1k, \qquad \|g_k\|_{C^{1}(\pa \Om_k \cap B_1)} \leq \frac1k, \qquad \|\Theta_{F_k}\|_{BMO} \leq \frac1k,
\end{equation}
and a corresponding local viscosity solution $u_k$ solving 
\begin{equation*}
\left\{
\begin{array}{ll}
F_k(x,D^2 u_k) = f_k & \text{in} \ \Om_k \cap B_1, \\
u_k = g_k & \text{on} \ \pa \Om_k \cap B_1,
\end{array}\right.
\end{equation*}
such that 
\begin{equation*}
\|u_k - \phi\|_{L^{\infty}(\Om_k\cap B_{1/2})} \geq \ve_0  \qquad  \text{for any} \ \phi \in C^{1,\tal}(B_1). 
\end{equation*}

Making use of uniform bounds in \eqref{5.4}, we can now use the H\"older estimate  upto the boundary  from \cite[Theorem 1.10]{MR2486925} to obtain
\begin{equation}
\label{eq4.8}
\begin{array}{rcl}
\|u_k\|_{C^{0,\alpha}(\Om_k \cap B_1)} &\leq & C(n,\La_0,\La_1,p) \lbr \|u_k\|_{L^{\infty}(\Om_k\cap B_1)} + \|g_k\|_{C^{0,1}(\pa\Om_k\cap B_1)} + \|f_k\|_{L^q(\Om_k\cap B_1)}\rbr\\
& \leq & C(n,\La_1,\La_1,q).
\end{array}
\end{equation}

We now use an idea similar to that in the proof of Lemma 4.1 in \cite{arXiv:1804.06697}. After flattening the boundary as in the proof of Lemma \ref{ext_lemma},  we extend $u_k$ to $B_1$ using \eqref{ext}  with $k_0=1$ and we still denote the extended function by  $u_k$. It is easy to see  that such an extension  ensures that $u_k$ is uniformly bounded in $C^{0,\al}(B_1)$.  As a consequence of the above estimates and hypothesis, we have the following convergence results:
\begin{enumerate}[(i)]
\item\label{conv1} Applying Arzela-Ascoli theorem to \eqref{eq4.8}, we see that there exists a function $u_{\infty} \in C^0(B_1)$ such that $u_k \rightarrow u_{\infty}$ uniformly in $B_1$. 
\item\label{conv2} From \eqref{5.4}, we see that $f_k \rightarrow 0$ as $k \rightarrow \infty$.
\item\label{conv3} From \eqref{5.4}, we see that $\Om_k \cap B_1 \rightarrow B_1^+$ as $k \rightarrow \infty$.
\item\label{conv4} From \eqref{5.4}, we see that $g_k \rightarrow 0$ as $k \rightarrow \infty$ which implies $u_{\infty} = 0$ on $B_1 \cap \{x_n = 0\}$. 
\item\label{conv5} Since $F_k$ is uniformly elliptic and $\|\Theta_{F_k}\|_{BMO} \rightarrow 0$, we see that $F_k(0,\cdot) \rightarrow F_{\infty}(\cdot)$ uniformly over compact subsets of $S(n)$.
\end{enumerate}

From the above convergence results along with an argument similar to \cite[Lemma 2.3]{MR1606359}, we get that $u_{\infty}$ is a $C^0$ viscosity solution of 
\begin{equation}\label{2.7}\left\{
\begin{array}{ll}
F_{\infty}(D^2 u_{\infty})=0 & \text{in}\ B_1^{+},\\
u_{\infty}=0\ & \text{on}\  x_n=0.
\end{array}\right.
\end{equation}

We can now make use of the estimate from Proposition \ref{bnd_reg} to get $u_{\infty}   \in C^{1,\beta}(\overline{B_{1/2}^+})$ for some $\beta \in (0,1)$. Now from the following expression of $u_\infty$ in $\{x_n<0\}$
\[  
u_{\infty}(x', x_n)=\sum_{i=1}^{2} c_i u_{\infty} ( x',- \frac{x_n}{i}),                   
 \]
 which follows from the uniform  convergence of extended $u_k$ to $u_\infty$, we conclude that

\begin{equation}\label{eq4.10}
\|u_{\infty}\|_{C^{1,\beta}(B_{1/2})} \leq C \|u_{\infty}\|_{C^{1,\beta}(\overline{B_{1/2}^+})}.
\end{equation}
Thus, from \ref{conv1} and \eqref{eq4.10}, we get $u_k \rightarrow u_{\infty}$ uniformly in $B_1$. In particular, this implies $$\|u_k - u_{\infty}\|_{L^{\infty}(B_1)} \rightarrow 0 \txt{as}k \rightarrow \infty,$$ which is a contradiction for large enough $k$. This completes the proof of the Lemma.
\end{proof}

We now have an important Corollary which proves a boundary affine approximation to viscosity solutions of \eqref{main}.

\begin{corollary}
\label{corollary5.2}
Let $\al$ be as in \eqref{al_alt}, then for any $u$ with $\|u\|_{L^{\infty}(\Om \cap B_1)}\leq 1$ solving \eqref{main} in the viscosity sense, there exist universal constants $\de_0 = \de_0(n,\La_0,\La_1,q,\al) \in (0,1)$ and $\la = \la(n,\La_0,\La_1,q,\de_0,\al) \in (0,1/2)$ such that if \eqref{eq4.3} holds for some $\de \in (0,\de_0)$, then there exists an affine function $L = bx_n$ on $B_{1/2}$ such that
\begin{equation*}
\|u - L\|_{L^{\infty}(\Om \cap B_{{\la}})} \leq  {\la}^{1+\al}.
\end{equation*}
\end{corollary}
\begin{proof}
From Lemma \ref{lemma4.3}, we see that for any $\ve >0$, there exists constant $\de$ (depending on $\ve$) and a function $u_{\infty} \in C^{1,\be}(\overline{\Om \cap B_{1/2}})$ such that if \eqref{eq4.3} holds,  then 
\begin{equation}\label{eq4.12}
\|u - u_{\infty} \|_{L^{\infty}(\Om \cap B_{1/2})} \leq \ve.
\end{equation}
Since $u_{\infty}$ solves \eqref{2.7}, we see from the observation $u_{\infty} = 0$ on $B_1 \cap \{x_n =0\}$, there exists an affine function $L = bx_x$ such that for all $r \in (0,1/2)$, there holds
\begin{equation}\label{eq4.13}
\|u_{\infty} - L\|_{L^{\infty}(\Om \cap B_r)} \leq {C}_{\holder} r^{1+\be}.
\end{equation}
Combining \eqref{eq4.12} and \eqref{eq4.13}, we get
\begin{equation}\label{eq4.14}
\begin{array}{rcl}
\|u - L\|_{L^{\infty}(\Om \cap B_r)} & \leq & \|u-u_{\infty} \|_{L^{\infty}(\Om \cap B_r)} + \|u_{\infty} - L\|_{L^{\infty}(\Om \cap B_r)} \\
& \leq & \|u-u_{\infty} \|_{L^{\infty}(\Om \cap B_{1/2})} + \|u_{\infty} - L\|_{L^{\infty}(\Om \cap B_r)} \\
& \leq & \ve + {C}_{\holder} r^{1+\be}.
\end{array}
\end{equation}

We now make the following choice of exponents:
\begin{itemize}
\item From the choice  $\al < \be$ (see \eqref{al_alt}) and  the observation $r^{1+\be} \rightarrow 0$ as $r \rightarrow 0$, there exist ${\la} \in (0,1/2)$ such that 
\begin{equation*}
{C}_{\holder} {\la}^{1+\be} = \frac12 {\la}^{1+\al}.
\end{equation*}
\item Now choose $\ve = \frac12 {\la}^{1+\al}$ which in turn fixes $\de$.
\end{itemize}
Using these constants in \eqref{eq4.14}, we get
\begin{equation*}
\|u - L\|_{L^{\infty}(\Om \cap B_{{\la}})} \leq {\la}^{1+\al},
\end{equation*}
which completes the proof of the Corollary.

\end{proof}

\subsection{Reductions}

By rotation, translation and scaling, we shall henceforth always assume everything is centred at $0 \in \pa \Om$ and $\Om$ satisfies the set up in Definition \ref{dini_domain} with $R_0=1$. From the observation that $u(x) - u(0) - \iprod{\nabla_xg(0)}{x}$ is also a solution of \eqref{main},  without loss of generality, we can assume that $u(0) = 0$ and $\nabla_xg(0)=0$.

For any fixed $r_0$, let us define the following rescaled functions:
\begin{equation}\label{rescale}
\begin{array}{ll}
u_{\mr}(x):= u(\mr x), &\qquad  g_{\mr}(x):= g(\mr x), \\
f_{\mr}(x):= \mr^2 f(\mr x), &\qquad \Ga_{\mr}(x') := \frac1{\mr} \Ga(\mr x').
\end{array}
\end{equation}

We make the following observations about the rescaled functions defined in \eqref{rescale}.
\begin{description}
\descitem{Observation 1}{obv1} Using \eqref{main}, the rescaled functions from \eqref{rescale} solves the following equation:
\begin{equation*}\label{6.3}\left\{
\begin{array}{ll}
F_{r_0}(x,D^2 u_{r_0})= \mr^2 F\lbr\mr x, \frac{1}{\mr^2}D^2 u_{\mr}\rbr=\mr^2 f(\mr x) & \text{in}\ \Om_{\mr} \cap B_1,\\
u_{\mr}(x)=g_{\mr}(x) \ & \text{on}\  \pa \Om_{\mr} \cap B_1.
\end{array}\right.
\end{equation*}
Here we have set $\Om_{\mr} :=\{x \in \RR^n: \mr x \in \Om\}$.
\descitem{Observation 2}{obv2} Computing the $C^1$ norm of $\Ga_{\mr}$, we get
\begin{equation}\label{eq4.20}
\begin{array}{rcl}
|\nabla  \Ga_{\mr}(x') - \nabla  \Ga_{\mr}(y')| & = & \abs{\frac1{\mr} \mr \nabla \Ga(\mr x')-\frac1{\mr} \mr \nabla \Ga(\mr y')}\\
& \leq & \om_{\Ga}(|\mr x' - \mr y'|)  \to 0  \txt{as} r_0 \to 0.
\end{array}
\end{equation}

\descitem{Observation 3}{obv3} Analogous to the calculation leading to \eqref{eq4.20}, we can compute the $C^1$ norm of $g_{\mr}$ to get
\begin{equation*}
 \om_{ \nabla g_{\mr}}(\cdot) \to 0  \txt{as} r_0\to 0.
\end{equation*}
\descitem{Observation 4}{obv4} For any $r<1$, the following estimate holds:
\begin{equation*}
\begin{array}{rcl}
r \lbr \fint_{B_r} |f_{\mr}|^q \ dx \rbr^{\frac{1}{q}} & \leq & r \lbr \fint_{B_r} |f_{\mr}|^n \ dx \rbr^{\frac{1}{n}}=  \mr \lbr \int_{B_{\mr r}} |f(x)|^n \ dx \rbr^{\frac{1}{n}}.
\end{array}
\end{equation*}
Therefore, if $\mr$ is small enough, then $r \lbr \fint_{B_r} |f_{\mr}|^q \ dx \rbr^{\frac{1}{q}}$ can be made  uniformly small for all $r<1$.
\descitem{Observation 5}{obv5}
\[
\Theta_{F_{r_0}} (x, y)= \Theta_F(r_0x, r_0y).
\]
\end{description}

\subsection{Boundary Approximation by Affine function}

Let $\de_0$ be as obtained in Corollary \ref{corollary5.2} which in-turn fixes $\la$ which depends on $\de_0$ and  is independent of $\de \in (0,\de_0)$. Now we choose another exponent $\tde$ (satisfying \eqref{eq7.21}) such that
\begin{equation}\label{eq4.23}
\tde < \de.
\end{equation}
Furthermore, we assume \eqref{eq4.3} holds with $\de=\tde$ which in view of the above discussion can be ensured by choosing $r_0$ small enough. In view of \dref{obv5}{Observation 5}, we note that the rescaling preserves the $BMO$-norm of the nonlinearity. Therefore by letting $F_{r_0}$ as our new $F$, $\Om_{r_0}$ as our new $\Om$, $u_{r_0}$ as our new $u$ and so on and finally by letting $\Theta \leq \tilde \delta$ where $\Theta$ is the bound on BMO norm of $F$  as in Theorem \ref{main_thm}, we can assume that \eqref{eq4.3} is satisfied for $r_0$ small enough. 
 
Let us define a few more functions:
\begin{description}
\descitem{$\md_{\one}$}{mod1} With $\al$ from \eqref{al_alt}, let us define 
\begin{equation*}\label{6.10}
\tom_1(r):= \max\{ \om_{\Ga}(r), \om_{g}(r), r^{\al}\}.
\end{equation*}
After normalizing and using Theorem \ref{mod_concave}, we can assume $\tom_1(\cdot)$ is concave and $\tom_1(1)=1$. We now define
\begin{equation*}
\om_1(r) = \tom_1(r^{\al}),
\end{equation*}
from which we see that $\om_1$ is $\al$-decreasing in the sense of Definition \ref{eta_decreasing}.
This modulus $\om_1$ is still $\dini$-continuous which can be ascertained by a simple of change of variables in \eqref{dini_form}.
\descitem{$\md_{\two}$}{mod2} We define
\begin{equation}
\label{eq4.26}
\om_2(r):= |\Om \cap B_1|^{\frac1n} r \lbr \fint_{\Om \cap B_r} |f|^q  \ dx \rbr^{\frac1q} = {C}_{\two} \ r \lbr \fint_{\Om \cap B_r} |f|^q  \ dx \rbr^{\frac1q}.
\end{equation}

\descitem{$\md_{\three}$}{mod3} With ${\la}$ obtained from Corollary \ref{corollary5.2} and $\tde$ from \eqref{eq4.23}, we define
\begin{equation}
\label{eq4.27}
\om_3(\la^k) := \frac{1}{\tde} \sum_{i=0}^k \om_1(\la^{k-i}) \om_2(\la^i).
\end{equation}

\descitem{$\md_{\four}$}{mod4} Finally, we define
\begin{equation}
\label{eq4.28}
\om_4(\la^k):= \max\{ \om_3(\la^k), \la^{\al k}\}.
\end{equation}
\end{description}

Let us first prove a preliminary Lemma that follows from \cite{MR3169795}.
\begin{lemma}
\label{lemma4.5}
The following bound holds:
\begin{equation}
\label{Cbnd}
\sum_{i=0}^{\infty} \om_4(\la^i) \leq {C}_{{bnd}}.
\end{equation}
\end{lemma}
\begin{proof}
From 	\eqref{eq4.28}, we have the trivial bound
\begin{equation*}
\label{eq4.30}
\sum_{i=0}^{\infty} \om_4(\la^i) \leq \sum_{i=0}^{\infty} \om_3(\la^i) + \sum_{i=0}^{\infty} \la^{i\al} \leq \sum_{i=0}^{\infty} \om_3(\la^i)  + {C}_{gm}.
\end{equation*}
In the above estimate, the constant ${C}_{gm}$ is the sum of geometric progression $\sum_{i=0}^{\infty} \la^{i\al}$. Hence, in order to prove the Lemma, it suffices to bound the first term, to do this, using \eqref{eq4.27} along with the fact that $\om_1(\cdot)$ is increasing, we get
\begin{equation}
\label{eq4.31}
\sum_{i=0}^{\infty} \om_3(\la^i) = \frac{1}{\tde} \sum_{i=0}^{\infty} \sum_{j=0}^{i} \om_1(\la^{i-j})\om_2(\la^j) = \frac{1}{\tde} \lbr \sum_{i=0}^{\infty} \om_1(\la^i) \rbr\lbr \sum_{i=0}^{\infty} \om_2(\la^i) \rbr.
\end{equation}

\begin{description}
\item[Estimate for $\sum_{i=0}^{\infty} \om_1(\la^i)$] Using the fact that $\om_1(\cdot)$ is $\dini$-continuous, we get 
\begin{equation}
\label{eq4.32}
\sum_{i=0}^{\infty} \om_1(\la^i) \leq \frac{1}{-\log \la} \sum_{i=1}^{\infty} \int_{\la^i}^{\la^{i-1}} \frac{\om_1(s)}{s} \ ds = \int_0^1 \frac{\om_1(s)}{s} \ ds < \infty.
\end{equation}

\item[Estimate for $\sum_{i=0}^{\infty} \om_2(\la^i)$] From \cite[Equation (3.4)]{MR3169795}, there exists a constant $\tilde{c}$ such that
\begin{equation}
\label{eq4.33}
\sum_{i=0}^{\infty} \om_2(\la^i) \leq \tilde{c} \int_0^2 \lbr \fint_{\Om \cap B_{\rho}} |f(x)|^q \ dx \rbr^{\frac{1}{q}} \ d\rho =: \tilde{c} {\bf \tilde I}_q^f(0,2).
\end{equation}
Further making use of  \cite[Equation (3.13)]{MR3169795}, we get
\begin{equation}
\label{eq4.344}
\sup_x {\bf \tilde I}_q^f(x,r) \leq \frac{1}{|B_1|^{\frac1n}} \int_{0}^{|B_r|} \left[ f^{**}(\rho) \rho^{\frac{q}{n}} \right]^{\frac{1}{q}} \ \frac{d\rho}{\rho}.
\end{equation}

\end{description}

Thus using \eqref{eq4.344}, \eqref{eq4.33} and \eqref{eq4.32} into \eqref{eq4.31}, we get the bound in \eqref{Cbnd}.
This completes the proof of the lemma.
\end{proof}

\begin{remark}
From the estimates \eqref{eq4.33} and \eqref{eq4.344}, we see that 
\begin{equation}
\label{eq4.355}
\sum_{i=0}^{\infty} \om_2(\la^i)  \leq  \tilde{c} \frac{1}{|B_1|^{\frac1n}} \int_{0}^{|B_2|} \left[ f^{**}(\rho) \rho^{\frac{q}{n}} \right]^{\frac{1}{q}} \ \frac{d\rho}{\rho}.
\end{equation}
\end{remark}

We need to prove another crucial bound given in the following lemma:
\begin{lemma}
\label{lemma4.6}
For a fixed $k$, the following bound holds:
\begin{equation*}
\la^{\al} \om_4(\la^k) \leq \om_4(\la^{k+1}).
\end{equation*}

\end{lemma}
\begin{proof}
From \eqref{eq4.28}, if $\om_4(\la^k) = \la^{\al k}$, then trivially, we get 
\begin{equation*}
\la^{\al} \om_4(\la^k) = \la^{\al(k+1)} \overset{\eqref{eq4.28}}{\leq} \om_4(\la^{k+1}).
\end{equation*}

Hence, we only have the consider the case $\om_4(\la^k) = \om_3(\la^k)$. In this case, we proceed as follows:
\begin{equation*}
\la^{\al} \om_3(\la^k) = \frac{1}{\tde} \sum_{i=0}^k \la^{\al}\om_1(\la^{k-i}) \om_2(\la^i) \overset{\redlabel{4.6a}{a}}{\leq}  \frac{1}{\tde} \sum_{i=0}^k \om_1(\la^{k+1-i}) \om_2(\la^i) \overset{\redlabel{4.6b}{b}}{\leq} \om_3(\la^{k+1}).
\end{equation*}
To obtain \redref{4.6a}{a}, we made use of the fact that $\om_1(\cdot)$ is $\al$-decreasing in the sense of Definition \ref{eta_decreasing} and to obtain \redref{4.6b}{b}, we used the definition of the expression in \eqref{eq4.27}.
\end{proof}

We now prove the following important lemma which gives a linear approximation to the solution $u$ of \eqref{main} at the boundary.

\begin{lemma}
\label{lemma4.7}
Let $\al$ be as in \eqref{al_alt}, then for any $u$ with $\|u\|_{L^{\infty}(\Om \cap B_1)}\leq 1$ solving \eqref{main} in the viscosity sense, there exist universal constants $\tde = \tde(n,\La_0,\La_1,q,\al) \in (0,\de_0)$ where $\de_0$ is from Corollary \ref{corollary5.2} such that if \eqref{5.4} holds with $\de=\tde$, then there exists a sequence of affine functions $L_k:= b_k x_n$ for $k=0,1,2\ldots$ such that the following holds:
\begin{description}
\descitem{A1}{a1} $|u - L_k| \leq \la^k \om_4(\la^k)$,
\descitem{A2}{a2} $|b_k - b_{k+1}| \leq {C}_{{b}} \om_4(\la^k) \txt{in} \Om \cap B_{\la^k}$.
\end{description}
Here $\la$ is from Corollary \ref{corollary5.2}.
\end{lemma}
\begin{proof}
The proof is by induction. In the case $k=0$, we have  $L_{\infty}=0$ and trivially get
\[
\|u\|_{L^{\infty}(\Om \cap B_1)} \leq 1 \overset{\eqref{eq4.28}}{\leq} \om_4(1).
\]

Let $k\in \NN$ be fixed and assume \descref{a1}{A1} and \descref{a2}{A2} holds for all $i =0,1,\ldots,k$. In order to prove the lemma, it suffices to show \descref{a1}{A1} and \descref{a2}{A2} holds  true for $k+1$. In order to do this, we rescale and make use of Corollary \ref{corollary5.2}.

Let us define the following rescaled functions defined for $x \in \tilde \Om \cap B_{1}$ where $\tilde \Om= \Om_{\la^k}$ (Note that in view of the discussion in \dref{obv2}{Observation 2}, $\tilde \Om$ is more "flat" than $\Om$ since $\la <1 $).
\begin{equation}\label{eq4.29}
\begin{array}{c}
v(x) := \frac{u(\la^k  x) - L_k( \la^k x) }{\la^k \om_4(\la^k)}, \qquad   \tF(x,M) := \frac{\la^k}{\om_4(\la^k)} F\lbr \la^k x, \frac{\om_4(\la^k)}{\la^k} M \rbr, \\
\tf(x):= \frac{\la^k}{\om_4(\la^k)} f(\la^k x), \qquad \tg(x):= \left.\frac{g(\la^k x)}{\la^k \om_4(\la^k)}\right|_{\pa \tilde \Om \cap B_{1}}, \qquad \tL_k(x):= \left.\frac{L_k(\la^k x)}{\la^k \om_4(\la^k)}\right|_{\pa \tilde \Om \cap B_{1}}.
\end{array}
\end{equation}
From the induction hypothesis on \descref{a1}{A1}, we see that $\|v\|_{L^{\infty}(\tilde \Om \cap B_{1})} \leq 1$. Furthermore, $v$ solves the equation
\begin{equation*}
\label{v_main}
\left\{\begin{array}{ll}
       \frac{\la^k}{\om_4(\la^k)} F\lbr \la^k x, \frac{\om_4(\la^k)}{\la^k} D^2 v\rbr = \frac{\la^k}{\om_4(\la^k)} f(\la^k x) & \txt{in} \tilde \Om \cap B_{1}\\
       v(x) = \frac{g(\la^kx) - L_k(\la^k x)}{\la^k \om_4(\la^k)}& \txt{on} \pa \tilde \Om \cap B_1.
       \end{array}
\right.
\end{equation*}

We shall now show that \eqref{5.4} is satisfied for some $\tde$ (this is where we make a choice for $\tde$) which enables us to apply Corollary \ref{corollary5.2} and complete the induction argument. Let us now check each of the terms in \eqref{5.4} are satisfied:
\vspace{.01in}
\begin{description}
\item[Bound for $\Theta_{\tF}$] We have the following bound:
\begin{equation*}
\Th_{\tF}(x,y)  =  \frac{\la^k}{\om_4(\la^k)} \sup_{M} \frac{\abs{F\lbr \la^k x, \frac{\om_4(\la^k)}{\la^k} M \rbr - F\lbr \la^k y, \frac{\om_4(\la^k)}{\la^k} M \rbr  }}{\|M\|} =  \Th_F(\la^kx, \la^k y).
\end{equation*}
In particular, the following estimate holds
\begin{equation*}
\Th_{\tF}(x,y) =  \Th_F(\la^kx, \la^k y).
\end{equation*}

\item[Bound for $\tf$]In this case, we get the sequence of estimates
\begin{equation*}
\begin{array}{rcl}
\lbr \int_{\tilde \Om \cap B_1} |\tf(x)|^q \ dx \rbr^{\frac1q} & = & \frac{\la^k}{\om_4(\la^k)} \lbr \fint_{\Om \cap B_{\la^k}} |f(x)|^q \ dx \rbr^{\frac1q}
 \overset{\eqref{eq4.26}}{=}  \frac{1}{{C}_{\two}}\frac{\om_2(\la^k)}{\om_4(\la^k)}\\
& \overset{\eqref{eq4.27}}{\leq} & \frac{1}{{C}_{\two}}\tde \frac{\om_2(\la^k)}{\om_1(1)\om_2(\la^k)}\\
& \overset{\text{\descref{mod2}{$\md_{\two}$}}}{\leq} & \frac{\tde}{{C}_{\two}}.
\end{array}
\end{equation*}

\item[Bound for $\tg$] By using the fact that $\nabla g$ has  modulus of continuity given by $\tde \om_1(\cdot)$ and  also that $\nabla g(0)=0$, we get from the mean value theorem that the following holds for any $y, z \in \pa \tilde \Om \cap B_1$:
\begin{equation}\label{eq4.35}
\begin{array}{rcl}
|\tg(y) - \tg(z)| & = & \frac{|g(\la^k y) - g(\la^k z)|}{\la^k \om_4(\la^k)}\\
& \leq & \frac{\tde \om_1(\la^k) |\la^k y - \la^k z|}{\la^k \om_4(\la^k)}\\
& \overset{\eqref{eq4.27},\eqref{eq4.28}}{\leq} & \tde |y-z|.
\end{array}
\end{equation}

\item[Bound for $\tL_k$] Again since  $\nabla \Ga$ has  modulus of continuity given by $\tde \om_1(\cdot)$ and $\nabla_{x'} \Ga(0)=0$,  we obtain by an analogous computation  and from the expression of $\tilde L_k$ as in \eqref{eq4.29} that the following holds for $y, z \in \pa \tilde \Om \cap B_1$,
\begin{equation}
\label{eq4.43}
\tilde L_k(y) - \tilde L_k(z) =  \frac{b_k(\la^ky_n-\la^kz_n)}{\la^k \om_1(\la^k)} = \frac{b_k  (\Ga(\la^k y') - \Ga(\la^k z'))}{\la^k \om_1(\la^k)} \leq b_k \tde |y-z|.
\end{equation}
Hence
\begin{equation*}
\|\tL_k\|_{C^{0,1}(\pa \tilde \Om \cap B_1)} \leq b_k \tde.
\end{equation*}

\item[Bound for $b_k$] From the induction hypothesis applied to \descref{a2}{A2}, we get
\begin{equation}
\label{eq4.45}
|b_k| \leq \sum_{i=1}^k |b_i - b_{i-1}| {\leq} {C}_{{b}}\sum_{i=0}^{k-1} \om_4(\la^i) \overset{\text{Lemma \ref{lemma4.5}}}{\leq} {C}_{{b}} {C}_{{bnd}}.
\end{equation}

\item[Choice of $\tde$] Combining the estimate from \eqref{eq4.35}, \eqref{eq4.43} and 	\eqref{eq4.45}, we get
\begin{equation*}
\left\| \frac{g(\la^k x) - L_k(\la^k x)}{\la^k \tom(\la^k)}\right\|_{C^{0,1}(\pa \Om \cap B_1)} \leq ({C}_{{b}}{C}_{{bnd}}+1) \tde.
\end{equation*}
We will choose $\tde$ smaller than $\de$ from \eqref{eq4.23} satisfying
\begin{equation}
\label{eq7.21}
\lbr {C}_{{b}}{C}_{{bnd}}+1 +\frac{1}{C_{\two}}\rbr \tde \leq \de.
\end{equation}

\end{description}

Thus all the hypothesis of Corollary \ref{corollary5.2} are satisfied and thus, we can find an affine function $\tL = b x_n$ in $\tilde \Om \cap B_{\la}$ such that
\begin{equation}
\label{eq4.48}
|v(x) - \tL(x)| \leq \la^{1+\al} \txt{for all} x \in \tilde \Om \cap B_{\la}.
\end{equation}
There also holds the following bound
\begin{equation}\label{7.23}
|b| \leq {C}. 
\end{equation}

In particular, if we define
\begin{equation*}
L_{k+1}(y) := L_k(y) + \la^k \om_4(\la^k) \tL\lbr\frac{y}{\la^k}\rbr \txt{for} y \in \Om \cap B_{\la^{k+1}},
\end{equation*}
then clearly after scaling back, we get 
\begin{equation*}
|u(y) - L_{k+1}(y)| \leq \la^{k+1} \la^{\al} \om_4(\la^k) \overset{\text{Lemma \ref{lemma4.6}}}{\leq} \la^{k+1} \om_4(\la^{k+1}).
\end{equation*}
Moreover, it follows from \eqref{7.23} and the expression of $L_{k+1}$ as above that
\[
|b_k -b_{k+1}| \leq C \om_4(\la^k).
\]
This completes the proof of the lemma. 
\end{proof}
 
With $L_k$ as in Lemma \ref{lemma4.7}, letting $k \rightarrow \infty$, we see that $L_k \rightarrow L_{\infty}$ for some linear function $L_{\infty}$. In the following lemma, we show that $L_{\infty}$ is an affine approximation to the viscosity solution $u$ of \eqref{main} at the boundary.
\begin{lemma}
\label{lemma4.9}
The linear function $L_{\infty}:= \lim_{k\rightarrow \infty} L_k$ is the affine approximation of $u$ on $\pa \Om$. In particular, given any $x_0 \in \pa \Om \cap B_{1/2}$, then there exists a modulus of continuity $K(\cdot)$ such that 
\begin{equation*}
|u(x) - L_{x_0}(x)| \leq C_{\text{aff}} |x-x_0| K(|x-x_0|).
\end{equation*}
Moreover, $K(\cdot)$ can be chosen to  $\al$-decreasing in the sense of Definition \ref{eta_decreasing} with $\al$ as in \eqref{al_alt}. \end{lemma}
 \begin{proof}
Without loss of generality,  we can assume that $x_0= 0$ and also that we are in the setup of  Lemma \ref{lemma4.3}. Now  let $x \in \Om \cap B_1$ such that $|x| \approx \la^k$ with $\la$  coming from Corollary \ref{corollary5.2}. In particular, we pick a point satisfying $\la^k \leq |x|\leq 2 \la^k$ for some $k \in \NN$. We have the following sequence of estimates:
\begin{equation}
\label{eq4.58}
\begin{array}{rcl}
|u(x) - L_{\infty}(x)| & \leq & |u(x) - L_k(x)|+ |L_k(x) - L_{\infty}(x)|\\
& \overset{\redlabel{a.7.34}{a}}{\leq} & \la^k \om_4(\la^k) + \sum_{i = 0}^{\infty} \la^k |b_{k+i} - b_{k+i+1}|\\
& \overset{\redlabel{b.7.34}{b}}{\leq} & \la^k \om_4(\la^k) + {C}_{{b}} \la^k\sum_{i = k}^{\infty} \om_4(\la^i),
\end{array}
\end{equation}
where to obtain \redref{a.7.34}{a}, we made use of \descref{a1}{A1} from Lemma \ref{lemma4.7} and to obtain \redref{b.7.34}{b}, we made use of \descref{a2}{A2} from Lemma \ref{lemma4.7}. 

From \eqref{eq4.28} and \eqref{eq4.27}, for a fixed $i \in \NN$, we get
\begin{equation*}
\om_4(\la^i)  \leq  \frac{1}{\tde} \sum_{j=0}^{i/2} \om_1(\la^{i-j}) \om_2(\la^j) +\frac{1}{\tde} \sum_{j=i/2}^{i} \om_1(\la^{i-j}) \om_2(\la^j) + \la^{i\alpha}.
\end{equation*}
Recall that $\om_1(\cdot)$ is monotone from \descref{mod1}{$\md_{\one}$} and making use of \eqref{eq4.355}, we get 
\begin{equation}
\label{eq4.60}
\begin{array}{rcl}
\sum_{i=k}^{\infty} \om_4(\la^i) & \leq & \frac{C}{\tde} \sum_{i=k}^{\infty} \om_1(\la^{i/2}) + \sum_{i=k}^{\infty} \la^{i\al} + \frac{1}{\tde} \sum_{i=k}^{\infty} \sum_{j=i/2}^i \om_1(\la^{i-j}) \om_2(\la^j)\\
& = & I + II + III.
\end{array}
\end{equation}

Before we estimate each of the terms of terms of \eqref{eq4.60}, we note that $\tde$ is fixed. Let us also define
\begin{equation}
\label{def_K}
\begin{array}{l}
K_1(\ve) :=  \sup_{a\geq 0}\int_{a}^{a+\ve^{1/2}} \frac{\om_1(t)}{t} \ dt, \quad K_2(\ve) := \ve^{\al}, \quad K_3(\ve) := \sup_{a \geq 0} \int_a^{a+\ve} \left[ f^{**}(\rho) \rho^{\frac{q}{n}} \right]^{\frac1q} \ \frac{d\rho}{\rho}.
\end{array}
\end{equation}

\begin{description}
\item[Estimate for $I$:] We estimate as follows:
\begin{equation}
\label{eq4.62}
I  \leq   C  \int_{0}^{\la^{\frac{k}{2}}} \frac{\om_1(t)}{t} \ dt \overset{\eqref{def_K}}{\leq} C K_1(\la^k).
\end{equation}
From \eqref{eq4.62} and the choice $|x| \approx \la^k$, it is easy to see that  $I \rightarrow 0$ as $|x| \rightarrow 0$. 
\item[Estimate for $II$:] We use the standard formula for  Geometric progressions to get
\begin{equation}
\label{eq4.63}
II \leq C \la^{k\al} \overset{\eqref{def_K}}{=} C K_2(\la^{k}).
\end{equation}
From \eqref{eq4.63} and the choice $|x| \approx \la^k$, it is easy to see that  $II \rightarrow 0$ as $|x| \rightarrow 0$. 

\item[Estimate for $III$:] In this case, we get
\begin{equation}
\label{eq4.64}
\begin{array}{rcl}
III & \leq & C \lbr \sum_{i=k/2}^{\infty} \om_2(\la^i)\rbr \lbr \sum_{i=1}^{\infty} \om_1(\la^i) \rbr \overset{\eqref{eq4.32}}{\leq} C \lbr \sum_{i=k/2}^{\infty} \om_2(\la^i)\rbr \\
& \overset{\redlabel{7.38a}{a}}{\leq} & C \int_0^{\la^{kn/2}} \left[ f^{**}(\rho) \rho^{\frac{q}{n}} \right]^{\frac1q} \ \frac{d\rho}{\rho}  \overset{\eqref{def_K}}{\leq}   C K_3(\la^k).
\end{array}
\end{equation}
To obtain \redref{7.38a}{a}, we made use of the estimates \cite[Equations (3.4) and (3.13)]{MR3169795} and the fact that since $\la<1$, we have
\[
\la^{nk/2} \leq \la^k\ \text{since $n \geq 2$}
\]
and hence 
\[
\int_0^{\la^{kn/2}} \left[ f^{**}(\rho) \rho^{\frac{q}{n}} \right]^{\frac1q} \ \frac{d\rho}{\rho} \leq \int_0^{\la^{k}} \left[ f^{**}(\rho) \rho^{\frac{q}{n}} \right]^{\frac1q} \ \frac{d\rho}{\rho}
\]

 From \eqref{eq4.64} and the choice $|x| \approx \la^k$, it is easy to see that  $III \rightarrow 0$ as $|x| \rightarrow 0$. 
\end{description}

\noindent \textbf{Claim}: Without loss of generality, we can assume that $K_i(\cdot)$'s  are $\al$-decreasing in the sense of Definition \ref{eta_decreasing} with $\al$ as in \eqref{al_alt}.

To prove the claim, we proceed as follows:
\begin{description}[leftmargin=*]
\item[$\al$-decreasing property of $K_1(\cdot)$] From the fact that $\om_1(\cdot)$ is a modulus of continuity and concave, we have that $K_1(\cdot)$ satisfies all the properties of Definition \ref{def_modulus} and hence is also a modulus of continuity. Using Theorem \ref{mod_concave}, without loss of generality, we can assume $K_1(\cdot)$ is also concave. Now replacing $K_1(s)$ with $K_1(s^{\al})$ if necessary, we can also assume $K_1(\cdot)$ is $\al$-decreasing in the sense of Definition \ref{eta_decreasing}.

\item[$\al$-decreasing property of $K_2(\cdot)$] This follows trivially from the definition of $K_2(\cdot)$ in \eqref{def_K}.

\medskip

\item[$\al$-decreasing property of $K_3(\cdot)$] From Definition \ref{def_modulus}, it is easy to see that $K_3(\cdot)$ is a modulus of continuity. Using Theorem \ref{mod_concave}, without loss of generality, we can assume $K_3(\cdot)$ is also concave. Now replacing $K_3(s)$ with $K_3(s^{\al})$, we can also assume $K_3(\cdot)$ is $\al$-decreasing in the sense of Definition \ref{eta_decreasing}.
\end{description}


With the new $K_i(\cdot)$'s which are now $\al$-decreasing,  we define
\begin{equation}
\label{def_K_combined}
  K(r) := K_1(r) + K_2(r)+K_3(r),
\end{equation}
which is again $\al$-decreasing.  Combining \eqref{eq4.62}, \eqref{eq4.63}, \eqref{eq4.64} and \eqref{eq4.60} along with \eqref{eq4.58} and making use the choice $|x| \approx \la^k$, we get 
\begin{equation*}
|u(x) - L_{\infty}(x)| \leq C \la^k K(\la^k) = C |x| K(|x|).
\end{equation*}
This completes the proof of the lemma.
\end{proof}

\begin{remark}
The $\al$-decreasing property of $K(\cdot)$  although not important in the proof of the above lemma, but  nevertheless it is crucially used  in the  proof of the main result when the interior and the boundary estimates are combined. 
\end{remark}




\section{Proof of the Main Theorem}
\label{section5}
In order to combine the interior regularity estimates  proved in \cite[Theorem 1.1]{MR3169795} with our boundary estimates, we need the following rescaled version of the interior estimates. 
\begin{proposition}
\label{prop5.1}
Let $u$ be a local viscosity solution of 
\begin{equation*}
F(x,D^2 u) = f(x) \txt{in} B_r \txt{for some} r \in (0,1). 
\end{equation*}
Then with modulus function $K(\cdot)$ as given in \eqref{def_K_combined}, there exists a universal constant $\Theta_0\in(0,1)$ and $C>0$ such that if $F$ has $\Theta_0$-BMO coefficients, then the following estimate holds:
\begin{equation*}
|\nabla u(0)|  \leq  \frac{C}{r} \lbr ||u||_{L^{\infty}(B_r)} + r K(r) \rbr.
\end{equation*}
Analogously, for any $y \in B_{r/2}$, there holds
\begin{equation*}
|\nabla u(y) - \nabla u(0)| \leq C \lbr K(|y|) + \frac{\|u\|_{L^{\infty}(B_r)}}{r^{1+\alpha}} |y|^{\al} + |y|^{\al} \rbr.
\end{equation*}

\end{proposition}
\begin{proof}

We will first recall a scale invariant version of the interior estimates. Define
\begin{equation*}
A(r):= \max\{ r^{1+\al}, \|u\|_{L^{\infty}(B_r)}\} \txt{and} v(x):= \frac{u(rx)}{A(r)},
\end{equation*}
where $\al$ is the minimum of the exponent from \eqref{al_alt} and \cite[Theorem 1.2]{MR3169795}. 
It is easy to see that $\|v\|_{L^{\infty}(B_1)} \leq 1$ and  $v$ solves
\begin{equation*}
\tF(x,M):=\frac{r^2}{A(r)} F\lbr rx, \frac{A(r)}{r^2}D^2 v\rbr = \frac{r^2}{A(r)} f(rx) \txt{in} B_1.
\end{equation*}
Since $\al < 1$, we have  $r^2 \leq r^{1+\al} \leq A(r)$.
From \dref{obv5}{Observation 5}, we see that the rescaled problem has the same BMO-coefficients. Hence using either the estimates from before or from \cite[Theorem 1.2]{MR3169795}, we get
\begin{equation*}
|\nabla v(0)|     \overset{\eqref{def_K_combined}}{\leq}  C \lbr 1 + \frac{r}{A(r)} K(r) \rbr.
\end{equation*}
. Analogously, from \cite[Theorem 1.3]{MR3169795} or from our estimates  specialized to the interior case, we also get
\begin{equation*}
|\nabla v(x) - \nabla v(0)| \leq C \lbr \frac{r}{A(r)} K(r|x|) + |x|^{\al} \rbr\txt{for any} x \in B_{1/2}.
\end{equation*}
Rescaling back to $u$, we get
\begin{equation*}
|\nabla u(0)| \leq  \frac{C}{r} \lbr \|u\|_{L^{\infty}(B_r)}  + rK(r)\rbr,
\end{equation*}
where we used  $r^{\al} \leq K(r)$. 

Analogously, for $y \in B_{r/2}$ (using $y=rx$), there holds
\begin{equation*}
|\nabla u(y) - \nabla u(0)| \leq C \lbr K(|y|) + \frac{\|u\|_{L^{\infty}(B_r)}}{r^{1+\alpha}} |y|^{\al} + |y|^{\al} \rbr,
\end{equation*}
which completes the proof of the proposition.
\end{proof}

The next lemma establishes that "$\nabla u$" is continuous at the boundary. 

\begin{lemma}\label{lemma5.1}
Given any two points $y,z \in \pa \Om \cap B_{1/2}$,  there exists a universal constant $C$ such that the following estimate holds:
\begin{equation*}
|\nabla L_y - \nabla L_z|\leq C K(|y-z|).
\end{equation*}
Here $L_y$ and $L_z$ denotes the linear function constructed in Lemma \ref{lemma4.9} at the boundary points $y$ and $z$ respectively and $K(\cdot)$ is the modulus defined in \eqref{def_K_combined}.
\end{lemma}
\begin{proof}
Let $|y-z|=r$ and choose a "non tangential"  point $x \in \Om$ such that $|x-y| \approx r$ and $|x-z| \approx r$. Furthermore, let $B_{\beta r}(x) \subset \Om$ for a universal $\be$ which can be chosen independent of $r$ and depending only on the Lipschitz character of $\Om$.

We see that  $v := u - L_y$ solves the same equation from \eqref{main}, thus from Proposition \ref{prop5.1}, we get
\begin{equation}
\label{eq5.12}
|\nabla v(x)| = |\nabla u(x) - \nabla L_y| \leq \frac{C}{r} \lbr \|u - L_y\|_{L^{\infty}(B_{\beta r}(x))} + r K(r) \rbr.
\end{equation}
From the boundary regularity estimate in Lemma \ref{lemma4.9}, we have
\begin{equation}
\label{eq5.13}
\|u - L_y\|_{L^{\infty}(B_{\be r}(x)} \leq C r K(r).
\end{equation}
Thus combining \eqref{eq5.12} and \eqref{eq5.13}, we get
\begin{equation}
\label{eq.12}
|\nabla u(x) - \nabla L_y| \leq \frac{C}{r} \lbr r K(r) + rK(r) \rbr  \leq C K(r).
\end{equation}
Likewise, since $|x-z| \approx r$, we also get
\begin{equation}
\label{eq.13}
|\nabla u(x) - \nabla L_z| \leq \frac{C}{r} \lbr r K(r) + rK(r) \rbr  \leq C K(r).
\end{equation}
Combining \eqref{eq.12} and \eqref{eq.13} with an application of triangle inequality, we get
\begin{equation*}
|\nabla L_y - \nabla L_z| \leq C K(r) = C K(|y-z|),
\end{equation*}
which completes the proof of the lemma.
\end{proof}

\subsection{Proof of Theorem \ref{main_thm}}
 Let $y,z \in \Om \cap B_{\frac14}$ be given.
We shall denote the points $y_0 \in \pa \Om$ and $z_0 \in \pa \Om$ to be the points such that the following holds:
\begin{equation}
\label{eq5.17}
d(y,y_0) = \min d(y, \pa \Om) \txt{and} d(z,z_0) = \min d(z,\pa \Om).
\end{equation}

Without loss of generality, let us assume that 
\begin{equation}
\label{5.18}
\de:= d(y,\pa \Om) = \max\{d(y,\pa \Om), d(z,\pa \Om)\},
\end{equation}
and split the proof into two cases.
\begin{description}
\item[Case  $d(y,z) \leq \frac{\de}{2}$] From \eqref{5.18}, we see that  $z \in B_{\de}(y) \subset \Om$. Using the notation from \eqref{eq5.17}, let 
us consider the function $v:= u-L_{y_0}$ which still solves the same problem from \eqref{main}. Then from the rescaled estimates in Proposition \ref{prop5.1}, we get
\begin{equation}
\label{eq5.21}
\begin{array}{rcl}
|\nabla v(y) - \nabla v(z)| & = & |\nabla u(y) - \nabla u(z)|\\
& \leq & C \lbr K(|y-z|) +  \frac{\|u-L_{y_0}\|_{L^{\infty}(B_{\delta})}}{\de} \frac{|y-z|^{\al}}{\de^{\al}} +  |y-z|^{\al}\rbr.
\end{array}
\end{equation}
Using Lemma \ref{lemma4.9}, we know that  $\|u-L_{y_0}\|_{L^{\infty}(B_{\delta})} \leq\de K(\de)$ and from \eqref{def_K_combined}, we have the bound $|y-z|^{\al} \leq K(|y-z|)$ which combined with \eqref{eq5.21} gives
\begin{equation}
\label{eq5.22}
|\nabla u(y) - \nabla u(z)| \leq C \lbr K(|y-z|) + \frac{K(\de)}{\de^{\al}} |y-z|^{\al}\rbr.
\end{equation}
Since  $K(\cdot)$ is $\al$-decreasing, therefore this  implies
\begin{equation}
\label{eq5.23}
\frac{K(\de)}{\de^{\al}} |y-z|^{\al} \leq K(|y-z|).
\end{equation}
Thus, combining \eqref{eq5.22} and \eqref{eq5.23}, we get
\begin{equation*}
|\nabla u(y) - \nabla u(z)| \leq C K(|y-z|)
\end{equation*}

\item[Case $d(y,z) > \frac{\de}{2}$] Using the notation from \eqref{eq5.17} and making use of triangle  inequality along with \eqref{5.18}, we have 
\begin{equation}
\label{eq5.25}
d(y_0,z) \leq d(y,y_0) + d(y,z)  = \de + d(y,z) <  3 d(y,z).
\end{equation}
From the choice of $z_0$ in \eqref{eq5.17}, we also get
\begin{equation}
\label{eq5.26}
d(z_0,z) \leq d(y_0,z) \overset{\eqref{eq5.25}}{\leq} 3 d(y,z)
\end{equation}

Combining \eqref{eq5.25} and \eqref{eq5.26}, we get
\begin{equation}
\label{eq5.27}
d(z_0,y_0) \leq d(z_0,z) + d(z,y) + d(y,y_0) \leq 7 d(y,z).
\end{equation}

We now have the following sequence of estimates:
\begin{equation*}
\begin{array}{rcl}
|\nabla u(y) - \nabla u(z)| & \leq &  |\nabla u(y) - \nabla L_{y_0}| + |\nabla L_{y_0} - \nabla L_{z_0}| + |\nabla L_{z_0} - \nabla u(z)| \\
& \overset{\text{Lemma \ref{lemma4.9}}}{\leq} & C K(|y-y_0|) + |\nabla L_{y_0} - \nabla L_{z_0}| + C K_ (|z-z_0|) \\
& \overset{\text{Lemma \ref{lemma5.1}}}{\leq} & CK(|y-z|) + C K(|y_0-z_0|) \\
& \overset{\eqref{eq5.27}}{\leq} & C K(C|y-z|).
\end{array}
\end{equation*}
\end{description}
This completes the proof of the Theorem.

\providecommand{\bysame}{\leavevmode\hbox to3em{\hrulefill}\thinspace}
\providecommand{\MR}{\relax\ifhmode\unskip\space\fi MR }
\providecommand{\MRhref}[2]{%
  \href{http://www.ams.org/mathscinet-getitem?mr=#1}{#2}
}
\providecommand{\href}[2]{#2}

\end{document}